\documentclass{amsart}
\usepackage{amsmath}
  \usepackage{paralist}
  \usepackage{amssymb}
 \usepackage{amsthm}
\usepackage[colorlinks=true]{hyperref}
\hypersetup{urlcolor=blue, citecolor=red}

\newtheorem{theorem}{Theorem}[section]

\newtheorem{lemma}[theorem]{Lemma}

\theoremstyle{definition}
\newtheorem{definition}[theorem]{Definition}
\newtheorem{remark}[theorem]{Remark}

\numberwithin{equation}{section}

\title[Orbital stability of standing waves]{Orbital stability of standing waves
for Schr\"{o}dinger type equations with slowly decaying linear
potential}

\author[Xinfu Li, Junying Zhao]{}

\thanks{Email Addresses:  lxylxf@tjcu.edu.cn (XL), lxyzhjy@tjcu.edu.cn (JZ)}

\begin{document}
\maketitle

\centerline{\scshape Xinfu Li and Junying Zhao}
\medskip
{\footnotesize
 \centerline{School of Science, Tianjin University of Commerce, Tianjin 300134,
 China}}

\bigskip

\begin{abstract}
In this paper, kinds of Schr\"{o}dinger type equations with slowly
decaying linear  potential and power type or convolution type
nonlinearities are considered. By using the concentration
compactness principle, the sharp Gagliardo-Nirenberg inequality and
a refined estimate of the linear operator, the existence and orbital
stability of standing waves in $L^2$-subcritical and $L^2$-critical
cases are established in a systematic way.

\medskip

\textbf{2010 Mathematics Subject Classification}: 35Q41, 35B35,
35J20.

\textbf{Keywords}: orbital stability, concentration compactness
principle, standing waves, slowly decaying potential,
Schr\"{o}dinger type equations.
\end{abstract}

\section{Introduction and main results}

\setcounter{section}{1} \setcounter{equation}{0}

In this paper, we consider the following  nonlinear Schr\"{o}dinger
type equations with slowly decaying linear potential
\begin{equation}\label{e1.1}
\begin{cases}
i\partial_tu+\Delta u+V(x)u+f(u)=0, \quad (t,x)\in
\mathbb{R}\times\mathbb{R}^{N},\\
u(0,x)=u_0(x)\in H^1(\mathbb{R}^N),\quad x\in \mathbb{R}^{N},
\end{cases}
\end{equation}
where $N\geq 1$, $u: \mathbb{R}\times\mathbb{R}^{N}\to \mathbb{C}$
is a complex valued function, $V(x)=\gamma|x|^{-\alpha}$, $\gamma\in
\mathbb{R}$, $0<\alpha<\min\{2,N\}$, and $f: \mathbb{C}\to
\mathbb{C}$ is assumed to be one of the following five types:

\textbf{Type 1}  $f(u)=|u|^{p-1}u$ with $1<p< (N+2)/(N-2)_+$, where
$1/(N-2)_+=1/(N-2)$ if $N\geq 3$ and $1/(N-2)_+=+\infty$ if $N=1,
2$;

\textbf{Type 2}  $f(u)=(I_\beta\ast|u|^q)|u|^{q-2}u$ with
$0<\beta<N$, $1+\beta/N<q< (N+\beta)/(N-2)_+$ and $I_{\beta}$ being
the Riesz potential defined for every $x\in \mathbb{R}^N \setminus
\{0\}$ by
\begin{equation*}
I_{\beta}(x)=\frac{\Gamma(\frac{N-\beta}{2})}{\Gamma(\frac{\beta}{2})\pi^{N/2}2^\beta|x|^{N-\beta}}
\end{equation*}
with $\Gamma$ denoting the Gamma function (see \cite{Riesz1949AM},
P.19);

\textbf{Type 3}  $f(u)=|u|^{p_1-1}u+|u|^{p_2-1}u$ with $1<p_1<p_2<
(N+2)/(N-2)_+$;

\textbf{Type 4}
$f(u)=(I_\beta\ast|u|^{q_1})|u|^{{q_1}-2}u+(I_\beta\ast|u|^{q_2})|u|^{{q_2}-2}u$
with $0<\beta<N$ and $1+\beta/N<q_1<q_2< (N+\beta)/(N-2)_+$;

\textbf{Type 5} $f(u)=(I_\beta\ast|u|^{q})|u|^{{q}-2}u+|u|^{p-1}u$
with $0<\beta<N$, $1+\beta/N<q< (N+\beta)/(N-2)_+$ and $1<p<
(N+2)/(N-2)_+$.

The Schr\"{o}dinger operator $i\partial_t+\Delta$ arises in various
physical contexts such as nonlinear optics and plasma physics, see
\cite{Chiao-Garmire-Townes-1964,Sulem-Sulem 2007,Zakharov-1972}. The
nonlinearity enters due to the effect of changes in the field
intensity on the wave propagation characteristics of the medium. The
potential $V(x)$ can be thought of as modeling inhomogeneities in
the medium. In particular, the operator $-\Delta-\frac{\gamma}{|x|}$
with Coulomb potential  provides a quantum mechanical description of
the Coulomb force between two charged particles and corresponds to
having an external attractive long-range potential due to the
presence of a positively charged atomic nucleus, see
\cite{Messiah-1961,Series-1957}. For the study of Schr\"{o}dinger
equation with a harmonic potential, we refer the readers to Zhang
\cite{Zhang JSP 2000, Zhang CPDE 2005}.

By a standing wave, we mean a solution of (\ref{e1.1}) in the form
$u(t,x)=e^{i\omega t}\varphi(x)$, where $\omega\in \mathbb{R}$ is a
constant and   $\varphi\in H^1(\mathbb{R}^N)$ satisfies the
stationary equation
\begin{equation}\label{e1.2}
-\Delta
\varphi-\frac{\gamma}{|x|^\alpha}\varphi+\omega\varphi=f(\varphi),\quad
x\in \mathbb{R}^N.
\end{equation}
For $f(u)$ is of Type 1, the existence of  ground state to
(\ref{e1.2}) was studied by  Fukaya and Ohta \cite{Fukaya-Ohta-2018}
and Fukuizumi and Ohta \cite{Fukuizumi-Ohta-2003}, and the
uniqueness of the positive radial solution with $\alpha=1$ was
studied by Benguria and Jeanneret \cite{Benguria-Jeanneret 1986}.
For $f(u)$ is of Type 2, the nonexistence, existence and uniqueness
of positive solution to (\ref{e1.2}) were studied in
\cite{{Cheng-Yang 2018},{Lieb-Simon 1977},{Lions 1981},{Lions
1984}}.

When $\gamma=0$ and  $f(u)$ is of Type 1, (\ref{e1.1}) is invariant
under the scaling transform
\begin{equation}\label{e1.5}
u(t,x)\to \lambda^{\frac{2}{p-1}}u(\lambda^2t,\lambda x).
\end{equation}
If $p =p_c:=1+ {4}/{N}$, the transform keeps the mass invariant and
the nonlinearity is called $L^2$-critical. In this case, Cazenave
\cite{Cazenave 2003} proved that the ground state solution of
(\ref{e1.2}) is orbitally stable for all $\omega>0$ if $p <p_c$,
while is unstable for all $\omega>0$ if $p_c<p<(N+2)/(N-2)$. The
instability of the bound state solution with  $p=p_c$ was proved by
Weinstein \cite{Weinstein 1983}. When $f(u)$ is  of Type 2, the
transform
\begin{equation}\label{e1.4}
u(t,x)\to \lambda^{\frac{\beta+2}{2q-2}}u(\lambda^2 t,\lambda x)
\end{equation}
keeps (\ref{e1.1}) invariant and  $q =q_c:= 1+ (2+\beta)/{N}$ is the
$L^2$-critical exponent. In this case, Cazenave and Lions
\cite{Cazenave-Lions 1982} showed the existence and orbital
stability of standing waves for $N=3,\ q=2$ and $\beta=2$. Recently,
in the $L^2$-subcritical case, that is, $1 +\beta/N < q < 1
+(2+\beta)/{N}$, Wang et al. \cite{Wang-Sun-Lv 2016} studied the
orbital stability of standing waves to (\ref{e1.1}). When $f(u)$ is
of Type 3,  we refer to \cite{{Fukuizumi 2003},{Soave 2018}} for the
existence, orbital stability and strong instability of standing
waves and \cite{{Bellazzini-Siciliano 2011},{Chen-Guo 2007},{Kikuchi
2007},{Liu-Shi 2018}} for that of Type 5.

When $\gamma\neq0$, it is easy to see that equation (\ref{e1.1})
does not enjoy the scaling invariance as well as the pace
translation invariance. In this case, when $f(u)$ is of Type 1,
Fukuizumi and Ohta
 studied the existence and orbital stability of
standing waves for $1<p<1+4/N$ in \cite{Fukuizumi-Ohta-2003} and the
strong instability of standing waves for $\gamma>0$ and
$1+4/N<p<(N+2)/(N-2)$ was studied by Fukaya and Ohta in
\cite{Fukaya-Ohta-2018}. The authors in \cite{{Cazenave-1989},{Dinh
2018 power},{Guo et al.2018},{Miao-Zhang-Zheng-2018}} studied the
local well-posedness, global existence, scattering and finite time
blowup of the solutions to (\ref{e1.1}). When $f(u)$ is of Type 2,
if $\beta=2$, $q=2$, $\alpha=1$ and $\gamma>0$, Ginibre and Velo
\cite{Ginibre-Velo 1980} obtained the global existence of the
solution to (\ref{e1.1}). If $N=3$ and $\beta=2$, Cazenave and Lions
\cite{Cazenave-Lions 1982} and Lions \cite{Lions 1984} established
the existence and orbital stability of standing waves for $q=2$, and
Cheng and Yang \cite{Cheng-Yang 2018} studied that for $q>2$ and
close to 2. More precisely, in (\ref{e1.2}), taking $\omega$ as a
fixed parameter, it is known that every solution $v\in
H^1(\mathbb{R}^N)$ to (\ref{e1.2}) is a critical point of the
functional $S_\omega$ defined by
\begin{equation*}
S_\omega(v):=\frac{1}{2}\int_{\mathbb{R}^N}\left(|\nabla
v|^2-\frac{\gamma}{|x|^\alpha}|v|^2+\omega|v|^2\right)dx-\int_{\mathbb{R}^N}F(v)dx,
\end{equation*}
where $F(s)=\int_0^s f(\tau) d\tau$. Denote by
$\mathcal{M_\omega^G}$ the set of all non-negative minimizers for
\begin{equation*}
\inf\{S_\omega(v):v\in H^1(\mathbb{R}^N)\backslash\{0\},\ v
\mathrm{\ is\ a\  solution\ of \ (\ref{e1.2})}\}.
\end{equation*}
For $f(u)$ is of  Type 1, $1<p<1+4/N$ and $\gamma\in \mathbb{R}$,
\cite{Fukuizumi-Ohta-2003} proved that there exist $\omega_0>0$ and
$\omega_*>\omega_0$ such that $\mathcal{M_\omega^G}$ is not empty
for any $\omega\in(\omega_0,\infty)$ and $e^{i\omega
t}\varphi_\omega(x)$ is orbitally stable for any
$\varphi_\omega(x)\in \mathcal{M_\omega^G}$ and
$\omega\in(\omega_*,\infty)$, by using a sufficient condition for
orbital stability, that is, the positive definite of the operator
$S''_\omega(\varphi_\omega)$. Similar results were obtained by
\cite{Cheng-Yang 2018} for Type 2 nonlinearities, $N=3$, $\beta=2$,
$q>2$ close to 2. Note that $\omega_0$ and $\omega_*$ are not given
explicitly, and moreover, they did not consider the $L^2$-critical
case. Hence, in this paper, we further discuss the existence and
orbital stability of standing waves of (\ref{e1.1}).

Note that  we may take $\omega$ as unknown in (\ref{e1.2}). Indeed,
for any $\rho>0$, if we define
\begin{equation*}
E(u)=\frac{1}{2}\int_{\mathbb{R}^N}\left(|\nabla
u|^2-\frac{\gamma}{|x|^\alpha}|u|^2\right)dx-\int_{\mathbb{R}^N}F(u)dx,
\end{equation*}
\begin{equation}\label{e1.3}
A_\rho=\inf_{u\in M_\rho}E(u),\ M_\rho=\{u\in
H^1(\mathbb{R}^N):\|u\|_2^2=\rho\}
\end{equation}
and
\begin{equation*}
G_\rho=\{u\in H^1(\mathbb{R}^N): u\in M_\rho,\ E(u)=A_\rho\}.
\end{equation*}
Then the Lagrange multiplier theorem implies that for any $u\in
G_\rho$, there exists $\lambda\in \mathbb{R}$ such that
\begin{equation*}
-\Delta u-\frac{\gamma}{|x|^\alpha}u-f(u)=-\lambda u.
\end{equation*}
Hence, $e^{i\lambda t}u(x)$ is a standing wave to (\ref{e1.1}) with
initial data $u_0(x)=u(x)$. One usually calls $e^{i\lambda t}u(x)$
the orbit of $u$. Moreover, if $u\in G_\rho$, then
$e^{i\theta}u(x)\in G_\rho$ for any $\theta\in \mathbb{R}$. In this
paper, we consider the orbital stability of  $G_\rho$. For this, we
give the following definition of orbital stability which is similar
to that in \cite{Trachanas-Zographopoulos-2015}.

\begin{definition}\label{def1.1}
The set $G_\rho$ is said to be orbitally stable if, for any
$\epsilon>0$, there exists $\delta>0$ such that for any initial data
$u_0$ satisfying
\begin{equation*}
\inf_{v\in G_\rho}\|u_0-v\|_{H^1(\mathbb{R}^N)}<\delta,
\end{equation*}
the corresponding solution $u$ to (\ref{e1.1}) satisfies
\begin{equation*}
\inf_{v\in G_\rho}\|u(t)-v\|_{H^1(\mathbb{R}^N)}<\epsilon
\end{equation*}
for all $t\geq 0$.
\end{definition}

\begin{remark}\label{rek1.1}
(1) Note that our definition of orbital stability is different from
that in \cite{Fukuizumi-Ohta-2003} for we do not know whether or not
$\mathcal{M_\omega^G}$ and $G_\rho$ are single point sets.

(2)  For the lack of scaling invariance, there is not direct
connection between $\mathcal{M_\omega^G}$ and $G_\rho$. However, for
equation
\begin{equation*}
\begin{cases}
i\partial_tu+\Delta u+(I_\beta\ast|u|^q)|u|^{q-2}u=0, \quad (t,x)\in
\mathbb{R}\times\mathbb{R}^{N},\\
u(0,x)=u_0(x)\in H^1(\mathbb{R}^N),\quad x\in \mathbb{R}^{N},
\end{cases}
\end{equation*}
there is some equivalence between $\mathcal{M_\omega^G}$ and
$G_\rho$, see \cite{Wang-Sun-Lv 2016}.
\end{remark}

To study the orbital stability of standing waves of (\ref{e1.1}), we
first make the following assumption on the local well-posedness.

\medskip

\textbf{Assumption A.} Let $N\geq 1$, $V(x)=\gamma|x|^{-\alpha}$,
$\gamma\in \mathbb{R}$, $0<\alpha<\min\{2,N\}$ and  $f(u)$ be one of
Types 1-5 such that the following local well-posedness holds for
(\ref{e1.1}):

For any $u_0\in H^1(\mathbb{R}^N)$, there exist $T_*,\ T^{*}\in
(0,\infty]$ and a unique solution $u(t)\in
C((-T_*,T^*),H^1(\mathbb{R}^N))\cap
C^{1}((-T_*,T^*),H^{-1}(\mathbb{R}^N))$ with $u(0)=u_0$ to
(\ref{e1.1}). If $T^*<\infty$ ($T_*<\infty$), then $\lim_{t\uparrow
T^*}\|u(t,x)\|_{H^1(\mathbb{R}^N)}=\infty$
($\lim_{t\downarrow-T_*}\|u(t,x)\|_{H^1(\mathbb{R}^N)}=\infty$).
Moreover, there is conservation of mass and energy, i.e.,
\begin{equation*}
\|u(t)\|_2^2=\|u_0\|_2^2, \ E(u(t))=E(u_0),  \ t\in(-T_*,T^*).
\end{equation*}

\begin{remark}\label{rmk1.1}
Motivated by \cite{Cazenave 2003}, \cite{Dinh 2018 power} and
\cite{Feng-Yuan 2015}, in Section 2 we give a proof for Assumption A
when $q\geq 2$.
\end{remark}

Under Assumption A, by using the concentration compactness
principle, the sharp Gagliardo-Nirenberg inequality and a refined
estimate of the linear operator $-\Delta-\gamma|x|^{-\alpha}$, we
can obtain the following theorems.

\begin{theorem}\label{thm1.1}
Let $N\geq 1$, $0<\alpha<\min\{2,N\}$,  $f(u)=|u|^{p-1}u$, $V(x)$
and $f(u)$ satisfy Assumption A.  Assume one of the following
conditions hold:

(1) $\gamma>0$, $1<p<1+4/N$, $\rho>0$;

(2) $\gamma>0$, $p=1+4/N$, $0<\rho<\|Q_p\|_2^2$, where $Q_p$ is the
unique positive radial solution of equation
\begin{equation*}
-\Delta Q+Q=Q^{p};
\end{equation*}
then $G_\rho$ is not empty and orbitally stable.
\end{theorem}

\begin{theorem}\label{thm1.3}
Let $N\geq 1$, $0<\alpha<\min\{2,N\}$, $\beta\in(0,N)$,
$f(u)=(I_\beta\ast|u|^q)|u|^{q-2}u$, $V(x)$ and $f(u)$ satisfy
Assumption A. Assume one of the following conditions hold:

(1) $\gamma> 0$, $1+{\beta}/{N}<q<1+(2+\beta)/N$, $\rho>0$;

(2) $\gamma> 0$, $q=1+(2+\beta)/N$, $0<\rho<\|W_q\|_2^2$, where
$W_q$ is a radically ground state solution of the elliptic equation
\begin{equation*}
-\Delta W+W=(I_\beta\ast|W|^q)|W|^{q-2}W;
\end{equation*}
then $G_\rho$ is not empty and orbitally stable.
\end{theorem}

\begin{theorem}\label{thm1.2}
Let $N\geq 1$, $0<\alpha<\min\{2,N\}$,
$f(u)=|u|^{p_1-1}u+|u|^{p_2-1}u$, $V(x)$ and $f(u)$ satisfy
Assumption A. Assume one of the following conditions hold:

(1) $\gamma> 0$, $1<p_1<p_2<1+4/N$, $\rho>0$;

(2) $\gamma> 0$, $1<p_1<p_2=1+4/N$, $0<\rho<\|Q_{p_2}\|_2^2$;\\
then $G_\rho$ is not empty and orbitally stable.
\end{theorem}

\begin{theorem}\label{thm1.4}
Let $N\geq 1$, $0<\alpha<\min\{2,N\}$, $\beta\in(0,N)$,
$f(u)=(I_\beta\ast|u|^{q_1})|u|^{q_1-2}u+(I_\beta\ast|u|^{q_2})|u|^{q_2-2}u$,
$V(x)$ and $f(u)$ satisfy Assumption A. Assume one of the following
conditions hold:

(1) $\gamma> 0$, $1+{\beta}/{N}<q_1<q_2<1+(2+\beta)/N$, $\rho>0$;

(2) $\gamma> 0$, $1+{\beta}/{N}<q_1<q_2=1+(2+\beta)/N$,
$0<\rho<\|W_{q_2}\|_2^2$;\\
then $G_\rho$ is not empty and orbitally stable.
\end{theorem}

\begin{theorem}\label{thm1.5}
Let $N\geq 1$, $0<\alpha<\min\{2,N\}$, $\beta\in(0,N)$,
$f(u)=(I_\beta\ast|u|^q)|u|^{q-2}u+|u|^{p-1}u$, $V(x)$ and $f(u)$
satisfy Assumption A. Assume one of the following conditions hold:

(1) $\gamma> 0$, $1<p<1+4/N$, $1+{\beta}/{N}<q<1+(2+\beta)/N$,
$\rho>0$;

(2) $\gamma> 0$, $1<p<1+4/N$, $q=1+(2+\beta)/N$,
$0<\rho<\|W_q\|_2^2$;

(3) $\gamma> 0$, $p=1+4/N$, $1+{\beta}/{N}<q<1+(2+\beta)/N$,
$0<\rho<\|Q_p\|_2^2$;

(4) $\gamma> 0$, $p=1+4/N$, $q=1+(2+\beta)/N$,
$(\sqrt{\rho}/\|Q_p\|_2)^{4/N}+(\sqrt{\rho}/\|W_q\|_2)^{(2\beta+4)/N}<1$;\\
then $G_\rho$ is not empty and orbitally stable.
\end{theorem}

We should point out that, among  the methods used in the study of
orbital stability of standing waves, the profile decomposition
method plays an  important role in recent studies, see
\cite{Bensouilah-Dinh-Zhu 2018}, \cite{Feng-Zhang 2018} and
\cite{Zhu JEE 2017}. In \cite{Bensouilah-Dinh-Zhu 2018}, the authors
considered a Schr\"{o}dinger equation with inverse-square potential,
i.e. (\ref{e1.1}) with $N\geq 3$, $\alpha=2$ and $\gamma<(N-2)^2/4$.
By using the equivalence of $\|\nabla u\|_2^2$ and
$\int_{\mathbb{R}^N}(|\nabla u|^2-\gamma|x|^{-2}|u|^2)dx$,
Bensouilah \cite{Bensouilah 2018} obtained the profile decomposition
of a bounded sequence in $H^1(\mathbb{R}^N)$ related to the problem,
and based of which, \cite{Bensouilah-Dinh-Zhu 2018} studied the
orbital stability of standing waves. However, there is not an
equivalence between $\|\nabla u\|_2^2$ and
$\int_{\mathbb{R}^N}(|\nabla u|^2-\gamma|x|^{-\alpha}|u|^2)dx$ for
$0<\alpha<\min\{2,N\}$ and we can not obtain the profile
decomposition in this case. But in view of that
$\int_{\mathbb{R}^N}(|\nabla u|^2-\gamma|x|^{-\alpha}|u|^2)dx$ can
be controlled by $\epsilon\|\nabla u\|_2^2$ and a function of
$\|u\|_2^2$ (see Lemma \ref{lem2.2}),  and by carefully examining
the application of concentration compactness principle in the study
of orbital stability of standing waves (see \cite{Cazenave-Lions
1982} and \cite{Lions 1984}), we can solve the problem by using the
concentration compactness principle in a systematic way. In fact,
the profile decomposition can be seen as another equivalent
description of the concentration compactness principle, see
\cite{Feng-Zhang 2018}.

\medskip

This paper is organized as follows. In Section 2, we give some
preliminaries. Section 3 is devoted to the proof of Theorems
\ref{thm1.1}-\ref{thm1.5}.

\medskip

\textbf{Notation.} Throughout this paper, we use the following
notation. $C>0$  stands for a constant that may be different from
line to line when it does not cause any confusion. The notation
$A\lesssim B$ means that $A\leq CB$ for some constant $C>0$.
$L^r(\mathbb{R}^N)$ with $1\leq r<\infty$ denotes the Lebesgue space
with the norms
$\|u\|_r=\left(\int_{\mathbb{R}^N}|u|^rdx\right)^{1/r}$.
 $ H^1(\mathbb{R}^N)$ is the usual Sobolev space with norm
$\|u\|_{H^1(\mathbb{R}^N)}=\left(\int_{\mathbb{R}^N}(|\nabla
u|^2+|u|^2)dx\right)^{1/2}$.  $B_R(x)$ denotes the ball in
$\mathbb{R}^N$ centered at $x$ with radius $R$.
$B_R^c(x)=\mathbb{R}^N\backslash B_R(x)$. $\chi_B(x)=1$ if $x\in B$,
and $=0$ if $x\not\in B$.

\section{Preliminaries}

\setcounter{section}{2} \setcounter{equation}{0}

The following generalized Gagliardo-Nirenberg inequality can be
found in \cite{Weinstein 1983}.

\begin{lemma}\label{lem2.1}
Let $N\geq 1$ and $0<\eta<4/(N-2)_+$, then the following sharp
Gagliardo-Nirenberg inequality
\begin{equation*}
\|u\|_{\eta+2}^{\eta+2}\leq
C_{GN}(\eta)\|u\|_2^{2+\eta(2-N)/2}\|\nabla u\|_2^{\eta N/2}
\end{equation*}
holds for any $u\in H^1(\mathbb{R}^N)$. The sharp constant
$C_{GN}(\eta)$ is
\begin{equation*}
C_{GN}(\eta)=\frac{2(\eta+2)}{4-(N-2)\eta}\left(\frac{4-(N-2)\eta}{N\eta}\right)^{N\eta/4}\frac{1}{\|Q_{\eta+1}\|_2^\eta},
\end{equation*}
where $Q_{\eta+1}$ is defined in Theorem \ref{thm1.1}.
\end{lemma}

Next, we give a refined estimate for the linear operator
$-\Delta-\gamma |x|^{-\alpha}$.

\begin{lemma}\label{lem2.2}
Let $N\geq 1$, $0<\alpha<\min\{2,N\}$ and $\gamma\in \mathbb{R}$.
Then for any $\epsilon>0$, there exists a constant
$\delta=\delta(\epsilon,\|u\|_2)>0$ such that
\begin{equation*}
\epsilon\int_{\mathbb{R}^N}|\nabla u|^2dx-\gamma
\int_{\mathbb{R}^N}\frac{|u|^2}{|x|^\alpha} dx\geq
-\delta(\epsilon,\|u\|_2)
\end{equation*}
for any $u\in H^1(\mathbb{R}^N)$.
\end{lemma}

\begin{proof}
It obviously  holds for $\gamma\leq 0$. Now we prove the lemma for
$\gamma>0$. By the domain decomposition, the H\"{o}lder inequality
and $0<\alpha<\min\{2,N\}$, we know
\begin{equation}\label{e2.1}
\begin{split}
\int_{\mathbb{R}^N}\frac{|u|^2}{|x|^\alpha}
dx&=\int_{B_1(0)}\frac{|u|^2}{|x|^\alpha}
dx+\int_{B_1^c(0)}\frac{|u|^2}{|x|^\alpha} dx\\
&\leq
\||x|^{-\alpha}\chi_{B_1(0)}\|_r\||u|^2\|_{r'}+\||x|^{-\alpha}\chi_{B_1^c(0)}\|_s\||u|^2\|_{s'}\\
&=C_1\|u\|_{2r'}^2+C_2\|u\|_{2s'}^2,
\end{split}
\end{equation}
where $1/r+1/r'=1$,  $1/s+1/s'=1$, $r<N/\alpha$, $s>N/\alpha$,
$N/\alpha-r$ and $s-N/\alpha$ are both sufficiently small. Hence,
$r'-N/(N-\alpha)>0$ and $N/(N-\alpha)-s'>0$ are small. By the
Gagliardo-Nirenberg inequality (Lemma \ref{lem2.1}), we know
\begin{equation}\label{e2.2}
\|u\|_{2r'}^2\leq C\|u\|_2^{N/r'-(N-2)}\|\nabla u\|_2^{N-N/r'}.
\end{equation}
Noting that $|N-N/r'-\alpha|$ is sufficiently small and
$0<\alpha<2$, and by using the Young inequality
\begin{equation*}
a^{1/q}b^{1/q'}\leq \frac{a}{q}+\frac{b}{q'},\ a, b>0, \ 1/q+1/q'=1,
\end{equation*}
we have for any $\epsilon_1>0$, there exists
$\delta_1=\delta_1(\epsilon_1,\|u\|_2)$ such that
\begin{equation}\label{e2.3}
\begin{split}
\|u\|_2^{N/r'-(N-2)}\|\nabla
u\|_2^{N-N/r'}&=\|u\|_2^{N/r'-(N-2)}\|\nabla
u\|_2^{2\frac{N-N/r'}{2}}\\
&\leq \epsilon_1\|\nabla u\|_2^2+\delta_1(\epsilon_1,\|u\|_2).
\end{split}
\end{equation}
The same estimates hold for $\|u\|_{2s'}^2$. In view of
(\ref{e2.1})-(\ref{e2.3}), we complete the proof.
\end{proof}

The following concentration compactness principle is cited from
Lemma III.1 in \cite{Lions 1984}.

\begin{lemma}\label{lem con-com}
Let $N\geq 1$ and $\{u_n\}_{n=1}^{\infty}$ be a bounded sequence in
$H^1( \mathbb{R}^N)$ satisfying:
\begin{equation*}
\int_{\mathbb{R}^N}|u_n|^2dx=\lambda,
\end{equation*}
where $\lambda>0$ is fixed. Then there exists a subsequence
$\{u_{n_k}\}_{k=1}^{\infty}$ satisfying one of the three
possibilities:

(i) (compactness) there exists $\{y_{n_k}\}_{k=1}^{\infty}\subset
\mathbb{R}^N$ such that $|u_{n_k}(\cdot+y_{n_k})|^2$ is tight, i.e.,
\begin{equation*}
\forall\ \epsilon>0, \ \exists\  R<\infty,\
\int_{B_R(y_{n_k})}|u_{n_k}(x)|^2dx\geq \lambda-\epsilon;
\end{equation*}

(ii) (vanishing) $\lim_{k\to
\infty}\sup_{y\in\mathbb{R}^N}\int_{B_R(y)}|u_{n_k}(x)|^2dx=0$ for
all $R<\infty$;

(iii) (dichotomy) there exists $\sigma\in(0,\lambda)$ such that for
any $\epsilon>0$, there exist $k_0\geq 1$, $R_1=R_1(\epsilon)>0$,
$\{y_{n_k}\}_{k=1}^{\infty}\subset \mathbb{R}^N$ and
$u_{n_k}^{(1)}$, $u_{n_k}^{(2)}$ bounded in $H^1(\mathbb{R}^N)$
satisfying for $k\geq k_0$:
\begin{equation*}
\begin{cases}
|u_{n_k}^{(1)}|,\  |u_{n_k}^{(2)}|\leq |u_{n_k}|;\\
\left\|u_{n_k}-(u_{n_k}^{(1)}+u_{n_k}^{(2)})\right\|_p\leq \delta_p(\epsilon)\to 0\ \mathrm{as}\ \epsilon\to 0\  \mathrm{for\ any}\  2\leq p<2N/(N-2)_+;\\
\left|\int_{\mathbb{R}^N}|u_{n_k}^{(1)}|^2dx-\sigma\right|\leq
\epsilon,\
\left|\int_{\mathbb{R}^N}|u_{n_k}^{(2)}|^2dx-(\lambda-\sigma)\right|\leq \epsilon;\\
\mathrm{Supp}\ u_{n_k}^{(1)}\subset B_{2R_1}(y_{n_k});\\
\mathrm{dist}(\mathrm{Supp}\ u_{n_k}^{(1)}, \ \mathrm{Supp}\
u_{n_k}^{(1)})\to \infty\ \mathrm{as}\ k\to \infty;\\
\liminf_{k\to \infty}\int_{\mathbb{R}^N}(|\nabla u_{n_k}|^2-|\nabla
u_{n_k}^{(1)}|^2-|\nabla u_{n_k}^{(2)}|^2)dx\geq 0.
\end{cases}
\end{equation*}
\end{lemma}

The following well-known Hardy-Littlewood-Sobolev inequality  can be
found in \cite{Lieb-Loss 2001}.

\begin{lemma}\label{lem HLS}
Let $N\geq 1$, $p$, $r>1$ and $0<\beta<N$ with
$1/p+(N-\beta)/N+1/r=2$. Let $u\in L^p(\mathbb{R}^N)$ and $v\in
L^r(\mathbb{R}^N)$. Then there exists a sharp constant
$C(N,\beta,p)$, independent of $u$ and $v$, such that
\begin{equation*}
\left|\int_{\mathbb{R}^N}\int_{\mathbb{R}^N}\frac{u(x)v(y)}{|x-y|^{N-\beta}}dxdy\right|\leq
C(N,\beta,p)\|u\|_p\|v\|_r.
\end{equation*}
If $p=r=\frac{2N}{N+\beta}$, then
\begin{equation*}
C(N,\beta,p)=C_\beta(N)=\pi^{\frac{N-\beta}{2}}\frac{\Gamma(\frac{\beta}{2})}{\Gamma(\frac{N+\beta}{2})}\left\{\frac{\Gamma(\frac{N}{2})}{\Gamma(N)}\right\}^{-\frac{\beta}{N}}.
\end{equation*}
\end{lemma}

\begin{remark}\label{rek1.31}
(1). By the Hardy-Littlewood-Sobolev inequality above, for any $v\in
L^s(\mathbb{R}^N)$ with $s\in(1,\frac{N}{\alpha})$, $I_\alpha\ast
v\in L^{\frac{Ns}{N-\alpha s}}(\mathbb{R}^N)$ and
\begin{equation*}
\|I_\alpha\ast v\|_{L^{\frac{Ns}{N-\alpha s}}}\leq C\|v\|_{L^s},
\end{equation*}
where $C>0$ is a constant depending only on $N,\ \alpha$ and $s$.

(2). By the Hardy-Littlewood-Sobolev inequality above and the
Sobolev embedding theorem, we obtain
\begin{equation}\label{e22.4}
\begin{split}
\int_{\mathbb{R}^N}(I_\beta\ast|u|^p)|u|^pdx\leq
C\left(\int_{\mathbb{R}^N}|u|^{\frac{2Np}{N+\beta}}dx\right)^{1+\beta/N}
\leq C\|u\|_{H^1(\mathbb{R}^N)}^{2p}
\end{split}
\end{equation}
for any $p\in \left[\frac{N+\beta}{N},\frac{N+\beta}{N-2}\right]$ if
$N\geq 3$ and $p\in \left[\frac{N+\beta}{N},+\infty\right)$ if $N=1,
2$, where $C>0$ is a constant depending only on $N,\ \beta$ and $p$.
\end{remark}

The following generalized Gagliardo-Nirenberg inequality for the
convolution problem can be found in \cite{Feng-Yuan 2015} and
\cite{Moroz-Schaftingen JFA 2013}.

\begin{lemma}\label{lem cGN}
Let $N\geq 1$, $0<\beta<N$, $1+\beta/N<p<(N+\beta)/(N-2)_+$, then
\begin{equation*}
\int_{\mathbb{R}^N}(I_\beta\ast|u|^p)|u|^pdx\leq C(\beta,p)\|\nabla
u\|_2^{Np-N-\beta}\|u\|_2^{N+\beta-Np+2p}.
\end{equation*}
The best constant $C(\beta,p)$ is defined by
\begin{equation*}
C(\beta,p)=\frac{2p}{2p-Np+N+\beta}\left(\frac{2p-Np+N+\beta}{Np-N-\beta}\right)^{(Np-N-\beta)/2}\|W_p\|_2^{2-2p},
\end{equation*}
where $W_p$ is defined in Theorem \ref{thm1.3}. In particular,  in
the $L^2$-critical case, i.e., $p=1+\frac{2+\beta}{N}$,
$C(\beta,p)=p\|W_p\|_2^{2-2p}$.
\end{lemma}

\begin{remark}\label{rmk2.3}
Note that $W_p$ may be not unique, but the ground state has the same
$L^2$-norm, see \cite{Feng-Yuan 2015}.
\end{remark}

In the end of this section, we prove the local well-posedness in the
energy space $H^1(\mathbb{R}^N)$ for (\ref{e1.1}) when $q\geq 2$. We
first recall the following result due to Cazenave (see Theorem 4.3.1
in \cite{Cazenave 2003}).

\begin{theorem}\label{thm2.1}
Let $N\geq 1$. Consider the following Cauchy problem
\begin{equation}\label{e2.5}
\begin{cases}
i\partial_tu+\Delta u+g(u)=0, \quad (t,x)\in
\mathbb{R}\times\mathbb{R}^{N},\\
u(0,x)=u_0(x)\in H^1(\mathbb{R}^N),\quad x\in \mathbb{R}^{N}.
\end{cases}
\end{equation}
Let $g=g_1+g_2+\cdots+g_k$ be such that the following assumptions
hold for every $i=1,2,\cdots, k$:

(A1) $g_i\in C(H^1(\mathbb{R}^N),H^{-1}(\mathbb{R}^N))$ and there
exists $G_i\in C^1(H^1(\mathbb{R}^N),\mathbb{R})$ such that
$g_i=G_i'$;

(A2) there exist $r,\rho\in [2,2N/(N-2)_+)$ if $N\geq 2$ ($r,\rho\in
[2,+\infty]$ if $N=1$) such that $g_i\in
C(H^1(\mathbb{R}^N),L^{\rho'}(\mathbb{R}^N))$ and such that for any
$M<\infty$ there exists $C(M)<\infty$ such that
\begin{equation*}
\|g_i(v)-g_i(w)\|_{\rho'}\leq C(M)\|v-w\|_{r}
\end{equation*}
for all $v,w\in H^1(\mathbb{R}^N)$ with
$\|v\|_{H^1(\mathbb{R}^N)}+\|w\|_{H^1(\mathbb{R}^N)}\leq M$;

(A3) for any $v\in H^1(\mathbb{R}^N)$,
$\mathrm{Im}(g_i(v)\bar{v})=0$ a.e. in $\mathbb{R}^N$.\\
Then for any $u_0\in H^1(\mathbb{R}^N)$, there exist $T_*,\ T^{*}\in
(0,\infty]$ and a unique solution $u(t)\in
C((-T_*,T^*),H^1(\mathbb{R}^N))\cap
C^{1}((-T_*,T^*),H^{-1}(\mathbb{R}^N))$ with $u(0)=u_0$ to
(\ref{e2.5}). If $T^*<\infty$ ($T_*<\infty$), then $\lim_{t\uparrow
T^*}\|u(t,x)\|_{H^1(\mathbb{R}^N)}=\infty$
($\lim_{t\downarrow-T_*}\|u(t,x)\|_{H^1(\mathbb{R}^N)}=\infty$).
Moreover, there is conservation of mass and energy:
\begin{equation*}
\|u(t)\|_2^2=\|u_0\|_2^2,\
E_g(u(t)):=\frac{1}{2}\int_{\mathbb{R}^N}|\nabla
u|^2dx+G_1(u)+\cdots+G_k(u)=E_g(u_0)
\end{equation*}
for all $t\in (-T_*,T^*)$.
\end{theorem}

Applying Theorem \ref{thm2.1}, we obtain the local well-posedness in
$H^1(\mathbb{R}^N)$ for (\ref{e1.1}).

\begin{theorem}\label{thm2.2}
 Let $N\geq 1$, $V(x)=\gamma|x|^{-\alpha}$,
$\gamma\in \mathbb{R}$, $0<\alpha<\min\{2,N\}$,  $f(u)$ be one of
Types 1-5, $q\geq 2$. Then Assumption A holds.
\end{theorem}

\begin{proof}
Denote $g(u)=\gamma
|x|^{-\alpha}u+|u|^{p-1}u+(I_\beta\ast|u|^q)|u|^{q-2}u:=g_1+g_2+g_3$.
Since $0<\alpha<\min\{2,N\}$, $\gamma |x|^{-\alpha}\in
L^r(\mathbb{R}^N)+L^{\infty}(\mathbb{R}^N)$ for some
$r>\max\{1,N/2\}$. By Example 3.2.11 in \cite{Cazenave 2003}, we
know that $g_1$ and $g_2$ satisfy assumptions (A1)-(A3) in Theorem
\ref{thm2.1}.

Since $2\leq q<2N/(N-2)_+$, it follows from (3.2.15) in
\cite{Cazenave 2003} that
\begin{equation*}
||w|^{q-2}w-|v|^{q-2}v|\leq C(|w|^{q-2}+|v|^{q-2})|v-w|.
\end{equation*}
By the mean value theorem,
\begin{equation*}
||w|^{q}-|v|^{q}|\leq C(|w|^{q-1}+|v|^{q-1})|v-w|.
\end{equation*}
Set $\rho=\frac{2Nq}{N+\alpha}$. By Lemma \ref{lem HLS}, Remark
\ref{rek1.31} and the H\"{o}lder inequality, we know
\begin{equation*}
\begin{split}
&\|g_3(v)-g_3(w)\|_{\rho'}\\
&=\|(I_\beta\ast|w|^q)(|w|^{q-2}w-|v|^{q-2}v)\|_{\rho'}
+\|(I_\beta\ast(|v|^{q-2}v)(|w|^{q}-|v|^{q})\|_{\rho'}\\
&\lesssim
\||w|^q\|_{\frac{2N}{N+\alpha}}\||w|^{q-2}w-|v|^{q-2}v\|_{\frac{q}{q-1}\frac{2N}{N+\alpha}}
+\||v|^{q-2}v\|_{\frac{q}{q-1}\frac{2N}{N+\alpha}}\||w|^q-|v|^q\|_{\frac{2N}{N+\alpha}}\\
&\lesssim
\|w\|_{\frac{2Nq}{N+\alpha}}^q\|(|w|^{q-2}+|v|^{q-2})|w-v|\|_{\frac{q}{q-1}\frac{2N}{N+\alpha}}\\
&\qquad+\|v\|_{\frac{2Nq}{N+\alpha}}^{q-1}\|(|w|^{q-1}+|v|^{q-1})|w-v|\|_{\frac{2N}{N+\alpha}}\\
&\lesssim
\|w\|_{\frac{2Nq}{N+\alpha}}^q(\|w\|_{\frac{2Nq}{N+\alpha}}^{q-2}+\|v\|_{\frac{2Nq}{N+\alpha}}^{q-2})\|w-v\|_{\frac{2Nq}{N+\alpha}}\\
&\qquad+\|v\|_{\frac{2Nq}{N+\alpha}}^{q-1}(\|w\|_{\frac{2Nq}{N+\alpha}}^{q-1}+\|v\|_{\frac{2Nq}{N+\alpha}}^{q-1})\|w-v\|_{\frac{2N}{N+\alpha}},
\end{split}
\end{equation*}
which implies that $g_3$ satisfies (A1)-(A3) in Theorem
\ref{thm2.1}. The proof is complete.
\end{proof}

\section{Proof of the main results}

\setcounter{section}{3} \setcounter{equation}{0}

\textbf{Proof of Theorem \ref{thm1.1}.} We will prove this theorem
in seven steps.

\textbf{Step 1.} We show that $A_\rho>-\infty$, that is, $A_\rho$ is
well defined. For any $u\in M_\rho$, we have $\|u\|_2^2=\rho$, and
by Lemmas \ref{lem2.1} and \ref{lem2.2}, we have for any
$\epsilon>0$,
\begin{equation}\label{e3.1}
\begin{split}
E(u)&=\frac{1}{2}\int_{\mathbb{R}^N}\left(|\nabla
u|^2-\frac{\gamma}{|x|^\alpha}|u|^2\right)dx-\frac{1}{p+1}\int_{\mathbb{R}^N}|u|^{p+1}dx\\
&\geq
\left(\frac{1}{2}-\frac{\epsilon}{2}\right)\int_{\mathbb{R}^N}|\nabla
u|^2dx-\delta(\epsilon,\|u\|_2)\\
&\qquad-\frac{1}{p+1}C_{NG}(p-1)\|u\|_2^{2+(p-1)(2-N)/2}\|\nabla
u\|_2^{(p-1)N/2}.
\end{split}
\end{equation}

For $1<p<1+4/N$, we have $0<(p-1)N/2<2$. Thus, by the Young
inequality, we obtain from (\ref{e3.1}) that
\begin{equation*}
\begin{split}
E(u)&\geq
\left(\frac{1}{2}-\frac{\epsilon}{2}\right)\int_{\mathbb{R}^N}|\nabla
u|^2dx-\delta(\epsilon,\|u\|_2)-\delta_1(\epsilon,\|u\|_2)-\frac{\epsilon}{2}\|\nabla u\|_2^2\\
&=\left(\frac{1}{2}-\epsilon\right)\|\nabla
u\|_2^2-\delta_2(\epsilon,\|u\|_2)\geq-\delta_2(\epsilon,\|u\|_2)>-\infty
\end{split}
\end{equation*}
by choosing $\epsilon=1/4$.

For $p=1+4/N$, we have
\begin{equation*}
\frac{(p-1)N}{2}=2,\ 2+\frac{(p-1)(2-N)}{2}=\frac{4}{N},\
\frac{1}{p+1}=\frac{N}{2N+4}
\end{equation*}
and $C_{GN}(p-1)=(N+2)/N\|Q_p\|_2^{-4/N}$. Thus, we obtain from
(\ref{e3.1}) that
\begin{equation*}
\begin{split}
E(u)&\geq
\left(\frac{1}{2}-\frac{\epsilon}{2}\right)\int_{\mathbb{R}^N}|\nabla
u|^2dx-\delta(\epsilon,\|u\|_2)-\frac{1}{2}\left(\frac{\|u\|_2}{\|Q_p\|_2}\right)^{4/N}\|\nabla
u\|_2^2\\
&=\frac{1}{2}\left(1-\left(\frac{\|u\|_2}{\|Q_p\|_2}\right)^{4/N}-\epsilon\right)\|\nabla
u\|_2^2-\delta(\epsilon,\|u\|_2)>-\infty
\end{split}
\end{equation*}
by choosing $\epsilon>0$ small enough since $\|u\|_2<\|Q_p\|_2$.

\textbf{Step 2.} For any $\rho>0$ there exists $C_1>0$ such that
$A_\rho\leq -C_1<0$. Indeed, for any $\varphi\in M_\rho$ and
$\lambda>0$, we define
$\varphi_\lambda(x)=\lambda^{N/2}\varphi(\lambda x)$. It is easy to
check that $\|\varphi_\lambda\|_2^2=\|\varphi\|_2^2=\rho$ and
\begin{equation*}
\begin{split}
E(\varphi_\lambda)=\frac{\lambda^2}{2}\int_{\mathbb{R}^N}|\nabla
\varphi|^2dx-\frac{\gamma
\lambda^\alpha}{2}\int_{\mathbb{R}^N}\frac{|\varphi|^2}{|x|^\alpha}dx-\frac{\lambda^{(p-1)N/2}}{p+1}\int_{\mathbb{R}^N}|\varphi|^{p+1}dx.
\end{split}
\end{equation*}
For $1<p<1+4/N$, we have $0<(p-1)N/2<2$. So by choosing $\lambda>0$
sufficiently small, we obtain $E(\varphi_\lambda)<0$ for any
$\gamma>0$. For $p=1+4/N$, we have $(p-1)N/2=2$. By choosing
$\varphi$ such that
\begin{equation*}
\frac{1}{2}\int_{\mathbb{R}^N}|\nabla
\varphi|^2dx<\frac{1}{p+1}\int_{\mathbb{R}^N}|\varphi|^{p+1}dx,
\end{equation*}
we have $E(\varphi_\lambda)<0$ for any $\lambda>0$ and $\gamma> 0$.

\medskip

Let $\{u_n\}_{n=1}^{\infty}\subset M_\rho$ be a minimizing sequence
of $E$, that is,
\begin{equation*}
\|u_n\|_2^2=\rho\ \ \mathrm{and}\ \ \lim_{n\to \infty}E(u_n)=A_\rho.
\end{equation*}
Similarly to Step 1, one can show that $\{u_n\}_{n=1}^{\infty}$ is
bounded in $H^1(\mathbb{R}^N)$. Then there exists a subsequence
$\{u_{n_k}\}_{k=1}^{\infty}$ such that one of the three
possibilities in Lemma \ref{lem con-com} holds.

\textbf{Step 3.} The vanishing case in Lemma \ref{lem con-com} does
not occur. Suppose by contradiction that
\begin{equation*}
\lim_{k\to \infty}\sup_{y\in
\mathbb{R}^N}\int_{B_1(y)}|u_{n_k}|^2dx=0.
\end{equation*}
By Lion's lemma, $u_{n_k}\to 0$ in $L^s(\mathbb{R}^N)$ as $k\to
\infty$ for all $s\in (2,2N/(N-2)_+)$. Hence,
\begin{equation*}
\int_{\mathbb{R}^N}|u_{n_k}|^{p+1}dx\to 0\quad \mathrm{and}\quad
\int_{\mathbb{R}^N}\frac{|u_{n_k}|^2}{|x|^\alpha}dx\to 0
\end{equation*}
by (\ref{e2.1}), and thus
\begin{equation*}
\lim_{k\to \infty}E(u_{n_k})=\lim_{k\to \infty}\frac{1}{2}\|\nabla
u_{n_k}\|_2^2\geq 0,
\end{equation*}
which contradicts $A_\rho<0$.

\textbf{Step 4}. We  show that the dichotomy case in Lemma \ref{lem
con-com} does not occur.

Firstly, according to (3) in \cite{Lions 1984}, we know
\begin{equation}\label{e3.6}
A_\rho\leq A_\eta+A_{\rho-\eta}^{\infty}\ \mathrm{for\ any\ }\eta\in
[0,\rho),
\end{equation}
where $A_0=0$,
$A_{\rho-\eta}^{\infty}=\inf_{M_{\rho-\eta}}E^{\infty}(u)$ and
$E^{\infty}(u)=\frac{1}{2}\|\nabla
u\|_2^2-\frac{1}{p+1}\|u\|_{p+1}^{p+1}$.

Next, we claim that
\begin{equation}\label{e3.7}
A_{\theta\eta}\leq \theta A_{\eta}\ \mathrm{for\ any\ }\eta\in
(0,\rho)\ \mathrm{and}\ \theta\in (1,\rho/\eta].
\end{equation}
Indeed, choose a sequence $\{\varphi_n\}_{n=1}^{\infty}\subset
M_{\eta}$ such that $\lim_{n\to \infty}E(\varphi_n)=A_\eta$. Then
$\|\sqrt{\theta}\varphi_n\|_2^2=\theta \eta$ and
\begin{equation*}
\begin{split}
A_{\theta\eta}&\leq \liminf_{n\to
\infty}E(\sqrt{\theta}\varphi_n)\\
&=\liminf_{n\to
\infty}\theta\left(\frac{1}{2}\int_{\mathbb{R}^N}|\nabla
\varphi_n|^2dx-\frac{\gamma}{2}\int_{\mathbb{R}^N}\frac{|\varphi_n|^2}{|x|^\alpha}dx\right)
-\frac{\theta^{\frac{p+1}{2}}}{p+1}\int_{\mathbb{R}^N}|\varphi_n|^{p+1}dx\\
&\leq \liminf_{n\to
\infty}\theta\left(\frac{1}{2}\int_{\mathbb{R}^N}|\nabla
\varphi_n|^2dx-\frac{\gamma}{2}\int_{\mathbb{R}^N}\frac{|\varphi_n|^2}{|x|^\alpha}dx
-\frac{1}{p+1}\int_{\mathbb{R}^N}|\varphi_n|^{p+1}dx\right)\\
&=\theta A_\eta
\end{split}
\end{equation*}
since $\theta>1$ and $p>1$. Hence the claim holds. It follows from
(\ref{e3.6}), (\ref{e3.7}) and Lemma II.1 in \cite{Lions 1984}, we
have
\begin{equation}\label{e3.8}
A_\rho<A_\eta+A_{\rho-\eta}\leq A_\eta+A_{\rho-\eta}^{\infty}\
\mathrm{for \ any\ }\eta\in(0,\rho).
\end{equation}

Finally, suppose by contradiction that (iii) in Lemma \ref{lem
con-com} holds. Denote
\begin{equation*}
\begin{split}
u_{n_k}=u_{n_k}^{(1)}+u_{n_k}^{(2)}+v_{n_k},\
d_{n_k}=\mathrm{dist}\{\mathrm{Supp}\ u_{n_k}^{(1)},\mathrm{Supp}\
u_{n_k}^{(2)}\},\\
\sigma_{n_k}=\int_{\mathbb{R}^N}|u_{n_k}^{(1)}|^2dx\ \ \mathrm{and}\
\  \eta_{n_k}=\int_{\mathbb{R}^N}|u_{n_k}^{(2)}|^2dx.
\end{split}
\end{equation*}
Then we may assume without loss of generality that
\begin{equation*}
d_{n_k}\to \infty,\ \sigma_{n_k}\to \bar{\sigma},\ \eta_{n_k}\to
\bar{\eta}
\end{equation*}
as $k\to \infty$ with $\bar{\sigma},\ \bar{\eta}\in (0,\rho)$,
$|\bar{\sigma}-\sigma|\leq \epsilon$ and
$|\bar{\eta}-(\rho-\sigma)|\leq \epsilon$.\\
By direct calculation, we obtain that
\begin{equation}\label{e3.21}
\begin{split}
\int_{\mathbb{R}^N}\frac{|u_{n_k}|^2}{|x|^\alpha}dx&=\int_{\mathbb{R}^N}\frac{|u_{n_k}^{(1)}+u_{n_k}^{(2)}+v_{n_k}|^2}{|x|^\alpha}dx\\
&=\int_{\mathbb{R}^N}\frac{|u_{n_k}^{(1)}|^2+|u_{n_k}^{(2)}|^2+|v_{n_k}|^2+2\mathrm{Re}(u_{n_k}^{(1)}\bar{v}_{n_k})+2\mathrm{Re}(u_{n_k}^{(2)}\bar{v}_{n_k})}{|x|^\alpha}dx\\
&=\int_{\mathbb{R}^N}\frac{|u_{n_k}^{(1)}|^2+|u_{n_k}^{(2)}|^2}{|x|^\alpha}dx+\delta(\epsilon),
\end{split}
\end{equation}
where $\delta(\epsilon)\to 0$ as $\epsilon\to 0$. If
$\{y_{n_k}\}_{k=1}^{\infty}$ is bounded, in view of (\ref{e2.1}), we
have
\begin{equation}\label{e3.22}
\int_{\mathbb{R}^N}\frac{|u_{n_k}^{(2)}|^2}{|x|^\alpha}dx=\int_{|x-y_{n_k}|\geq
d_{n_k}}\frac{|u_{n_k}^{(2)}|^2}{|x|^\alpha}dx\leq
\int_{|x-y_{n_k}|\geq d_{n_k}}\frac{|u_{n_k}|^2}{|x|^\alpha}dx \to 0
\end{equation}
as $k\to\infty$ since $d_{n_k}\to \infty$. Similarly, if
$\{y_{n_k}\}_{k=1}^{\infty}$ is unbounded, then
\begin{equation}\label{e3.23}
\int_{\mathbb{R}^N}\frac{|u_{n_k}^{(1)}|^2}{|x|^\alpha}dx=\int_{|x-y_{n_k}|\leq
2R_1}\frac{|u_{n_k}^{(1)}|^2}{|x|^\alpha}dx\leq
\int_{|x-y_{n_k}|\leq 2R_1}\frac{|u_{n_k}|^2}{|x|^\alpha}dx\to 0.
\end{equation}
By using the inequality
\begin{equation}\label{e2.4}
\left||\sum_{j=1}^{l}a_j|^r-\sum_{j=1}^{l}|a_j|^r\right|\leq
C\sum_{j\neq k}|a_j||a_k|^{r-1}
\end{equation}
with $r>1$ and $l\geq 2$, and in view of
$u_{n_k}^{(1)}u_{n_k}^{(2)}=0$, we obtain that
\begin{equation*}
\begin{split}
&\left|\int_{\mathbb{R}^N}(|u_{n_k}^{(1)}+u_{n_k}^{(2)}+v_{n_k}|^{p+1}-|u_{n_k}^{(1)}|^{p+1}-|u_{n_k}^{(2)}|^{p+1})dx\right|\\
&\leq \int_{\mathbb{R}^N}(|v_{n_k}|^{p+1}+C(|u_{n_k}^{(1)}||v_{n_k}|^p+|u_{n_k}^{(2)}||v_{n_k}|^p+|u_{n_k}^{(1)}|^p|v_{n_k}|+|u_{n_k}^{(2)}|^p|v_{n_k}|))dx\\
&=\delta(\epsilon),
\end{split}
\end{equation*}
which implies that
\begin{equation}\label{e3.24}
\begin{split}
\int_{\mathbb{R}^N}|u_{n_k}|^{p+1}dx=\int_{\mathbb{R}^N}|u_{n_k}^{(1)}|^{p+1}dx
+\int_{\mathbb{R}^N}|u_{n_k}^{(2)}|^{p+1}dx+\delta(\epsilon)
\end{split}
\end{equation}
for any $k\geq k_0$, where $\delta(\epsilon)\to 0$ as $\epsilon\to
0$. By (\ref{e3.21})-(\ref{e3.24}), we obtain
\begin{equation*}
\begin{split}
A_\rho&=\lim_{k\to\infty}\left(\frac{1}{2}\int_{\mathbb{R}^N}|\nabla
u_{n_k}|^2dx-\frac{\gamma}{2}\int_{\mathbb{R}^N}\frac{|u_{n_k}|^2}{|x|^\alpha}dx-\frac{1}{p+1}\int_{\mathbb{R}^N}|u_{n_k}|^{p+1}dx\right)\\
&\geq \liminf_{k\to
\infty}\left(\frac{1}{2}\int_{\mathbb{R}^N}|\nabla
u_{n_k}^{(1)}|^2dx-\frac{\gamma}{2}\int_{\mathbb{R}^N}\frac{|u_{n_k}^{(1)}|^2}{|x|^\alpha}dx
-\frac{1}{p+1}\int_{\mathbb{R}^N}|u_{n_k}^{(1)}|^{p+1}dx\right)+\delta(\epsilon)\\
&\quad+ \liminf_{k\to
\infty}\left(\frac{1}{2}\int_{\mathbb{R}^N}|\nabla
u_{n_k}^{(2)}|^2dx-\frac{\gamma}{2}\int_{\mathbb{R}^N}\frac{|u_{n_k}^{(2)}|^2}{|x|^\alpha}dx-\frac{1}{p+1}\int_{\mathbb{R}^N}|u_{n_k}^{(2)}|^{p+1}dx\right)\\
&\geq A_{\bar{\sigma}}+A_{\bar{\eta}}^{\infty}+\delta(\epsilon)\ \
\mathrm{or}\ \
A_{\bar{\sigma}}^{\infty}+A_{\bar{\eta}}+\delta(\epsilon).
\end{split}
\end{equation*}
Letting $\epsilon\to 0$, we obtain that $A_\rho\geq
A_{\sigma}+A_{\rho-\sigma}^{\infty}$ or $A_\rho\geq
A_{\sigma}^{\infty}+A_{\rho-\sigma}$, which contradicts
(\ref{e3.8}). Hence, (iii) in Lemma \ref{lem con-com} does not
occur.

\textbf{Step 5.} From Steps 3 and 4 and Lemma \ref{lem con-com},
there exist a subsequence $\{u_{n_k}\}_{k=1}^{\infty}$ and a
sequence $\{y_{n_k}\}_{k=1}^{\infty}\subset \mathbb{R}^N$ such that
for all $\epsilon>0$, there exists $R(\epsilon)>0$ such that for all
$k\geq 1$,
\begin{equation}\label{*}
\int_{B_{R(\epsilon)}(y_{n_k})}|u_{n_k}(x)|^2dx\geq
\lambda-\epsilon.
\end{equation}
Denote $\tilde{u}_{n_k}(\cdot)=u_{n_k}(\cdot+y_{n_k})$. Then there
exists $\tilde{u}$ such that $\tilde{u}_{n_k}\rightharpoonup
\tilde{u}$ weakly in $H^1(\mathbb{R}^N)$, $\tilde{u}_{n_k}\to
\tilde{u}$ strongly in $L_{\mathrm{loc}}^r(\mathbb{R}^N)$ with $r\in
[2,2N/(N-2)_+)$, which combine (\ref{*}) show that
\begin{equation*}
\int_{B_{R(\epsilon)}(0)}|\tilde{u}(x)|^2dx\geq \lambda-\epsilon.
\end{equation*}
Thus $\int_{\mathbb{R}^N}|\tilde{u}(x)|^2dx=\lambda$, i.e.,
$\tilde{u}_{n_k}\to \tilde{u}$ strongly in $L^2(\mathbb{R}^N)$. By
the Gagliardo-Nirenberg inequality, $\tilde{u}_{n_k}\to \tilde{u}$
strongly in $L^s(\mathbb{R}^N)$ for $s\in [2,2N/(N-2)_+)$.

Next, we claim that $\{y_{n_k}\}_{k=1}^{\infty}$ is bounded. If it
were not the case we deduce that
\begin{equation*}
\int_{\mathbb{R}^N}\gamma
|x|^{-\alpha}|u_{n_k}|^2dx=\int_{\mathbb{R}^N}\gamma
|x+y_{n_k}|^{-\alpha}|\tilde{u}_{n_k}(x)|^2dx\to 0
\end{equation*}
as $k\to \infty$. Hence $A_\rho\geq A_\rho^{\infty}$. On the other
hand, we know  that $A_{\rho}^{\infty}$ is attained by a nontrivial
function $v_\rho$. Indeed, Steps 1-4 hold for $\gamma=0$. Hence,
\begin{equation*}
\begin{split}
A_\rho^{\infty}&=\lim_{k\to \infty}E^{\infty}(u_{n_k})=\lim_{k\to \infty}E^{\infty}(\tilde{u}_{n_k})\\
&=\lim_{k\to \infty}\frac{1}{2}\int_{\mathbb{R}^N}|\nabla
\tilde{u}_{n_k}|^2dx-\frac{1}{p+1}\int_{\mathbb{R}^N}|\tilde{u}_{n_k}|^{p+1}dx\\
&\geq \frac{1}{2}\int_{\mathbb{R}^N}|\nabla
\tilde{u}|^2dx-\frac{1}{p+1}\int_{\mathbb{R}^N}|\tilde{u}|^{p+1}dx.
\end{split}
\end{equation*}
By the definition of $A_\rho^{\infty}$, we see that $\tilde{u}$ is a
minimizer of $A_\rho^{\infty}$. Consequently,
\begin{equation*}
A_\rho<A_\rho^{\infty}-\frac{1}{2}\int_{\mathbb{R}^N}\gamma|x|^{-\alpha}|v_\rho|^2dx<A_\rho^{\infty},
\end{equation*}
which contradicts $A_\rho\geq A_{\rho}^{\infty}$. Thus,
$\{y_{n_k}\}_{k=1}^{\infty}$ is bounded. So $u_{n_k}$ converges
strongly to some $u$ in $L^s(\mathbb{R}^N)$ for $s\in
[2,2N/(N-2)_+)$. Hence,
\begin{equation*}
\begin{split}
A_\rho&=\lim_{k\to \infty}E(u_{n_k})\\
&=\lim_{k\to
\infty}\frac{1}{2}\int_{\mathbb{R}^N}|\nabla
u_{n_k}|^2dx-\frac{\gamma}{2}\int_{\mathbb{R}^N}\frac{|u_{n_k}|^2}{|x|^\alpha}dx-\frac{1}{p+1}\int_{\mathbb{R}^N}|u_{n_k}|^{p+1}dx\\
&\geq \frac{1}{2}\int_{\mathbb{R}^N}|\nabla
u|^2dx-\frac{\gamma}{2}\int_{\mathbb{R}^N}\frac{|u|^2}{|x|^\alpha}dx-\frac{1}{p+1}\int_{\mathbb{R}^N}|u|^{p+1}dx.
\end{split}
\end{equation*}
By the definition of $A_\rho$, we see that $u$ is a minimizer of
$A_\rho$,  $\lim_{k\to \infty}\|\nabla u_{n_k}\|_2^2=\|\nabla
u\|_2^2$, and hence $u_{n_k}\to u$ in $H^1(\mathbb{R}^N)$.

\textbf{Step 6.} We show that under Assumption A, if $1<p<1+4/N$ and
$u_0\in H^1(\mathbb{R}^N)$ or $p=1+4/N$ and $\|u_0\|_2<\|Q_p\|_2$,
then the Cauchy problem (\ref{e1.1}) admits a global solution
$u(t)\in C(\mathbb{R},H^1(\mathbb{R}^N))\cap
C^1(\mathbb{R},H^{-1}(\mathbb{R}^N))$ with $u(0)=u_0$.

Indeed, by Assumption A, it suffices to bound $\|\nabla u(t)\|_2$ in
the existence time. By the conversation law, Lemmas \ref{lem2.1} and
\ref{lem2.2}, we know
\begin{equation}\label{e3.9}
\begin{split}
\|\nabla
u(t)\|_2^2&=2E(u(t))+\frac{2}{p+1}\|u(t)\|_{p+1}^{p+1}+\int_{\mathbb{R}^N}\frac{\gamma}{|x|^\alpha}|u(t)|^2dx\\
&\leq 2E(u(0))+\epsilon \|\nabla
u(t)\|_2^2+\delta(\epsilon,\|u(t)\|_2)\\
&\quad+\frac{2}{p+1}C_{NG}(p-1)\|u(t)\|_2^{2+(p-1)(2-N)/2}\|\nabla
u(t)\|_2^{(p-1)N/2}.
\end{split}
\end{equation}
Similarly to Step 1, for $p=1+4/N$, we have
\begin{equation*}
\|\nabla u(t)\|_2^2\leq 2E(u(0))+\epsilon \|\nabla
u(t)\|_2^2+\delta(\epsilon,\|u(t)\|_2)+\left(\frac{\|u(t)\|_2}{\|Q_p\|_2}\right)^{4/N}\|\nabla
u(t)\|_2^2.
\end{equation*}
Since $\|u(t)\|_2=\|u(0)\|_2<\|Q_p\|_2$, by choosing $\epsilon>0$
sufficiently small, the above estimate implies the boundedness of
$\|\nabla u(t)\|_2$. For $1<p<1+4/N$, we have
\begin{equation*}
\|\nabla u(t)\|_2^2\leq 2E(u(0))+\epsilon \|\nabla
u(t)\|_2^2+\delta(\epsilon,\|u(t)\|_2)+\epsilon \|\nabla
u(t)\|_2^2+\delta_1(\epsilon,\|u(t)\|_2)
\end{equation*}
and we arrive at the conclusion.

\textbf{Step 7.} We prove that $G_\rho$ is orbitally stable. Suppose
by contradiction that there exist sequences
$\{u_{0,n}\}_{n=1}^{\infty}\subset H^1(\mathbb{R}^N)$ and
$\{t_{n}\}_{n=1}^{\infty}\subset \mathbb{R}^+$ and a constant
$\epsilon_0>0$ such that for all $n\geq 1$,
\begin{equation}\label{e4.1}
\inf_{v\in G_\rho}\|u_{0,n}-v\|_{H^1(\mathbb{R}^N)}<\frac{1}{n}
\end{equation}
and
\begin{equation}\label{e4.2}
\inf_{v\in G_\rho}\|u_n(t_n)-v\|_{H^1(\mathbb{R}^N)}\geq \epsilon_0,
\end{equation}
where $u_n(t)$ is the solution to (\ref{e1.1}) with initial data
$u_{0,n}$.

We claim that there exists $v\in G_\rho$ such that
\begin{equation*}
\lim_{n\to \infty}\|u_{0,n}-v\|_{H^1(\mathbb{R}^N)}=0.
\end{equation*}
Indeed, by (\ref{e4.1}), there exists $\{v_n\}_{n=1}^{\infty}\subset
G_\rho$ such that
\begin{equation}\label{e4.3}
\|u_{0,n}-v_n\|_{H^1(\mathbb{R}^N)}<\frac{2}{n}.
\end{equation}
That $\{v_n\}_{n=1}^{\infty}\subset G_\rho$ implies that
$\{v_n\}_{n=1}^{\infty}\subset M_\rho$ is a minimizing sequence of
$E$. So by Steps 3-5, there exists $v\in G_\rho$ such that
\begin{equation}\label{e4.4}
\lim_{n\to \infty}\|v_n-v\|_{H^1(\mathbb{R}^N)}=0.
\end{equation}
Then the claim follows from (\ref{e4.3}) and (\ref{e4.4})
immediately. Hence,
\begin{equation*}
\lim_{n\to \infty}\|u_{0,n}\|_2^2=\|v\|_2^2=\rho,\ \ \lim_{n\to
\infty}E(u_{0,n})=E(v)=A_{\rho}.
\end{equation*}
By the conservation of mass and energy, we have
\begin{equation*}
\lim_{n\to \infty}\|u_n(t_n)\|_2^2=\rho,\ \ \lim_{n\to
\infty}E(u_n(t_n))=E(v)=A_{\rho}.
\end{equation*}
Similarly to  Step 1, $\{u_n(t_n)\}_{n=1}^{\infty}$ is bounded in
$H^1(\mathbb{R}^N)$. Set
\begin{equation*}
\tilde{u}_n=\frac{\sqrt{\rho}u_n(t_n)}{\|u_n(t_n)\|_2}.
\end{equation*}
Then $\|\tilde{u}_n\|_2^2=\rho$ and
\begin{equation*}
\begin{split}
E(\tilde{u}_n)&=\left(\frac{\sqrt{\rho}}{\|u_n(t_n)\|_2}\right)^2\left(\frac{1}{2}\int_{\mathbb{R}^N}|\nabla
u_n(t_n)|^2dx-\frac{\gamma}{2}\int_{\mathbb{R}^N}\frac{|u_n(t_n)|^2}{|x|^\alpha}dx\right)\\
&\qquad -\left(\frac{\sqrt{\rho}}{\|u_n(t_n)\|_2}\right)^{p+1}\frac{1}{p+1}\int_{\mathbb{R}^N}|u_n(t_n)|^{p+1}dx\\
&=\frac{\rho}{\|u_n(t_n)\|_2^2}
E(u_n(t_n))\\
&\qquad+\left(\left(\frac{\sqrt{\rho}}{\|u_n(t_n)\|_2}\right)^2-\left(\frac{\sqrt{\rho}}{\|u_n(t_n)\|_2}\right)^{p+1}\right)
\frac{1}{p+1}\int_{\mathbb{R}^N}|u_n(t_n)|^{p+1}dx,
\end{split}
\end{equation*}
which implies that
\begin{equation*}
\lim_{n\to
\infty}E(\tilde{u}_n)=\lim_{n\to\infty}E(u_n(t_n))=A_\rho.
\end{equation*}
Hence, $\{\tilde{u}_n\}_{n=1}^{\infty}\subset M_\rho$ is a
minimizing sequence of $E$, and by Steps 3-5, there exists
$\tilde{v}\in G_\rho$ such that
\begin{equation}\label{e3.25}
\lim_{n\to \infty}\|\tilde{u}_n-\tilde{v}\|_{H^1(\mathbb{R}^N)}=0.
\end{equation}
By the definition of $\tilde{u}_n$, we know
\begin{equation}\label{e3.26}
\lim_{n\to
\infty}\|\tilde{u}_n-u_n(t_n)\|_{H^1(\mathbb{R}^N)}=\lim_{n\to
\infty}\left(1-\frac{\sqrt{\rho}}{\|u_n(t_n)\|_2}\right)\|u_n(t_n)\|_{H^1(\mathbb{R}^N)}=0.
\end{equation}
(\ref{e3.25}) and (\ref{e3.26}) imply that
\begin{equation*}
\lim_{n\to \infty}\|u_n(t_n)-\tilde{v}\|_{H^1(\mathbb{R}^N)}=0,
\end{equation*}
which contradicts (\ref{e4.2}). The proof is complete.

\textbf{Proof of Theorem \ref{thm1.3}.} The proof follows from that
of Theorem \ref{thm1.1} by line to line, so we only point out the
differences.

\textbf{Step 1.} We show that $A_\rho>-\infty$. For any $u\in
M_\rho$, we have $\|u\|_2^2=\rho$, and by Lemmas \ref{lem2.2} and
\ref{lem cGN}, we have for any $\epsilon>0$,
\begin{equation}\label{e33.1}
\begin{split}
E(u)&=\frac{1}{2}\int_{\mathbb{R}^N}\left(|\nabla
u|^2-\frac{\gamma}{|x|^\alpha}|u|^2\right)dx-\frac{1}{2q}\int_{\mathbb{R}^N}(I_\beta\ast|u|^q)|u|^{q}dx\\
&\geq
\left(\frac{1}{2}-\frac{\epsilon}{2}\right)\int_{\mathbb{R}^N}|\nabla
u|^2dx-\delta(\epsilon,\|u\|_2)\\
&\qquad-\frac{1}{2q}C(\beta,q)\|u\|_2^{N+\beta-Nq+2q}\|\nabla
u\|_2^{Nq-N-\beta}.
\end{split}
\end{equation}

For $1+\beta/N<q<1+(\beta+2)/N$, we have $0<Nq-N-\beta<2$. Thus, by
the Young inequality, we obtain from (\ref{e33.1}) that
\begin{equation*}
\begin{split}
E(u)&\geq
\left(\frac{1}{2}-\frac{\epsilon}{2}\right)\int_{\mathbb{R}^N}|\nabla
u|^2dx-\delta(\epsilon,\|u\|_2)-\delta_1(\epsilon,\|u\|_2)-\frac{\epsilon}{2}\|\nabla u\|_2^2\\
&=\left(\frac{1}{2}-\epsilon\right)\|\nabla
u\|_2^2-\delta_2(\epsilon,\|u\|_2)\geq-\delta_2(\epsilon,\|u\|_2)>-\infty
\end{split}
\end{equation*}
by choosing $\epsilon=1/4$.

For $q=1+(\beta+2)/N$, we have $Nq-N-\beta=2$,\
$N+\beta-Nq+2q=(2\beta+4)/N$ and
$C(\beta,q)=q\|W_q\|_2^{2-2q}=q\|W_q\|_2^{-(2\beta+4)/N}$. Thus, we
obtain from (\ref{e33.1}) that
\begin{equation*}
\begin{split}
E(u)&\geq
\left(\frac{1}{2}-\frac{\epsilon}{2}\right)\int_{\mathbb{R}^N}|\nabla
u|^2dx-\delta(\epsilon,\|u\|_2)-\frac{1}{2}\left(\frac{\|u\|_2}{\|W_q\|_2}\right)^{(2\beta+4)/N}\|\nabla
u\|_2^2\\
&=\frac{1}{2}\left(1-\left(\frac{\|u\|_2}{\|W_q\|_2}\right)^{(2\beta+4)/N}-\epsilon\right)\|\nabla
u\|_2^2-\delta(\epsilon,\|u\|_2)>-\infty
\end{split}
\end{equation*}
by choosing $\epsilon>0$ small enough since $\|u\|_2<\|W_q\|_2$.

\textbf{Step 2.} For any $\rho>0$ there exists $C_1>0$ such that
$A_\rho\leq -C_1<0$. Indeed, for any $\varphi\in M_\rho$ and
$\lambda>0$, we define
$\varphi_\lambda(x)=\lambda^{N/2}\varphi(\lambda x)$. Then
$\|\varphi_\lambda\|_2^2=\|\varphi\|_2^2=\rho$ and
\begin{equation*}
\begin{split}
E(\varphi_\lambda)=\frac{\lambda^2}{2}\int_{\mathbb{R}^N}|\nabla
\varphi|^2dx-\frac{\gamma
\lambda^\alpha}{2}\int_{\mathbb{R}^N}\frac{|\varphi|^2}{|x|^\alpha}dx
-\frac{\lambda^{Nq-N-\beta}}{2q}\int_{\mathbb{R}^N}(I_\beta\ast|\varphi|^q)|\varphi|^qdx.
\end{split}
\end{equation*}
Similarly to Step 2 in the proof of Theorem \ref{thm1.1}, for any
$1+\beta/N<q\leq 1+(\beta+2)/N$, we can choose $\varphi$ and
$\lambda$ such that $E(\varphi_\lambda)<0$.

\medskip

Let $\{u_n\}_{n=1}^{\infty}\subset M_\rho$ be a minimizing sequence
of $E$, that is,
\begin{equation*}
\|u_n\|_2^2=\rho\ \ \mathrm{and}\ \ \lim_{n\to \infty}E(u_n)=A_\rho.
\end{equation*}
Then $\{u_n\}_{n=1}^{\infty}$ is bounded in $H^1(\mathbb{R}^N)$ and
there exists a subsequence $\{u_{n_k}\}_{k=1}^{\infty}$ such that
one of the three possibilities in Lemma \ref{lem con-com} holds.

\textbf{Step 3.} The vanishing case in Lemma \ref{lem con-com} does
not occur. Suppose by contradiction that
\begin{equation*}
\lim_{k\to \infty}\sup_{y\in
\mathbb{R}^N}\int_{B_r(y)}|u_{n_k}|^2dx=0.
\end{equation*}
Lion's lemma implies that $u_{n_k}\to 0$ in $L^s(\mathbb{R}^N)$ for
any $s\in (2,2N/(N-2)_+)$. Hence,
\begin{equation*}
\int_{\mathbb{R}^N}(I_\alpha\ast|u_{n_k}|^q)|u_{n_k}|^{q}dx\to
0\quad \mathrm{and}\quad
\int_{\mathbb{R}^N}\frac{|u_{n_k}|^2}{|x|^\alpha}dx\to 0
\end{equation*}
according to (\ref{e2.1}) and (\ref{e22.4}), and then
\begin{equation*}
\lim_{k\to \infty}E(u_{n_k})=\lim_{k\to
\infty}\frac{1}{2}\|u_{n_k}\|_2^2\geq 0,
\end{equation*}
which contradicts $A_\rho<0$.

\textbf{Step 4}. The dichotomy case in Lemma \ref{lem con-com} does
not occur. Similarly,  we have
\begin{equation*}
A_\rho\leq A_\eta+A_{\rho-\eta}^{\infty}\ \mathrm{for\ any\ }\eta\in
[0,\rho)
\end{equation*}
and
\begin{equation*}
A_{\theta\eta}\leq \theta A_{\eta}\ \mathrm{for\ any\ }\eta\in
(0,\rho)\ \mathrm{and}\ \theta\in (1,\rho/\eta],
\end{equation*}
where $A_0=0$,
$A_{\rho-\eta}^{\infty}=\inf_{M_{\rho-\eta}}E^{\infty}(u)$ and
$E^{\infty}(u)=\frac{1}{2}\|\nabla
u\|_2^2-\frac{1}{2q}\int_{\mathbb{R}^N}(I_\beta\ast|u|^q)|u|^qdx$.
Consequently,
\begin{equation}\label{e33.8}
A_\rho<A_\eta+A_{\rho-\eta}\leq A_\eta+A_{\rho-\eta}^{\infty}\
\mathrm{for \ any\ }\eta\in(0,\rho).
\end{equation}

Suppose by contradiction that (iii) in Lemma \ref{lem con-com}
holds. By using (\ref{e2.4}) and $u_{n_k}^{(1)}u_{n_k}^{(2)}=0$, we
obtain that
\begin{equation*}
\begin{split}
&\int_{\mathbb{R}^N}(I_\beta\ast|u_{n_k}|^q)|u_{n_k}|^{q}dx\\
&=\int_{\mathbb{R}^N}(I_\beta\ast|u_{n_k}^{(1)}+u_{n_k}^{(2)}+v_{n_k}|^q)|u_{n_k}^{(1)}+u_{n_k}^{(2)}+v_{n_k}|^{q}dx\\
&=\int_{\mathbb{R}^N}(I_\beta\ast(|u_{n_k}^{(1)}+u_{n_k}^{(2)}|^q+|v_{n_k}|^q+C(|u_{n_k}^{(1)}+u_{n_k}^{(2)}|^{q-1}
|v_{n_k}|+|u_{n_k}^{(1)}+u_{n_k}^{(2)}||v_{n_k}|^{q-1})))\\
&\qquad\times(|u_{n_k}^{(1)}+u_{n_k}^{(2)}|^q+|v_{n_k}|^q+C(|u_{n_k}^{(1)}+u_{n_k}^{(2)}|^{q-1}
|v_{n_k}|+|u_{n_k}^{(1)}+u_{n_k}^{(2)}||v_{n_k}|^{q-1}))dx\\
&=\int_{\mathbb{R}^N}(I_\beta\ast|u_{n_k}^{(1)}+u_{n_k}^{(2)}|^q)|u_{n_k}^{(1)}+u_{n_k}^{(2)}|^{q}dx+\delta(\epsilon)\\
&=\int_{\mathbb{R}^N}(I_\beta\ast|u_{n_k}^{(1)}|^q)|u_{n_k}^{(1)}|^{q}dx
+\int_{\mathbb{R}^N}(I_\beta\ast|u_{n_k}^{(2)}|^q)|u_{n_k}^{(2)}|^{q}dx\\
&\qquad
+2\int_{\mathbb{R}^N}(I_\beta\ast|u_{n_k}^{(1)}|^q)|u_{n_k}^{(2)}|^{q}dx
+\delta(\epsilon),
\end{split}
\end{equation*}
where $\delta(\epsilon)\to 0$ as $\epsilon\to0$. Since $0<\beta<N$
and $1+\beta/N<q<(N+\beta)/(N-2)_+$, we can choose a constant
$\delta>0$ sufficiently small such that
\begin{equation*}
0<\beta+\delta<N\ \ \mathrm{and}\ \ 1+\frac{\beta+\delta}{N}\leq q<
\frac{N+\beta+\delta}{(N-2)_+},
\end{equation*}
which combines with $d_{n_k}=\mathrm{dist}\{\mathrm{Supp}\
{u_{n_k}^{(1)}}, \mathrm{Supp}\ {u_{n_k}^{(2)}}\}\to \infty$ and
(\ref{e22.4}) gives that
\begin{equation*}
\begin{split}
&\int_{\mathbb{R}^N}\int_{\mathbb{R}^N}\frac{|u_{n_k}^{(1)}(x)|^q|u_{n_k}^{(2)}(y)|^q}{|x-y|^{N-\beta}}dxdy\\
&\qquad\leq
\frac{1}{d_{n_k}^\delta}\int_{\mathbb{R}^N}\int_{\mathbb{R}^N}\frac{|u_{n_k}(x)|^q|u_{n_k}(y)|^q}{|x-y|^{N-\beta-\delta}}dxdy\\
&\qquad\leq C
\frac{1}{d_{n_k}^\delta}\left(\int_{\mathbb{R}^N}|u_{n_k}|^{\frac{2Nq}{N+\beta+\delta}}dx\right)^{1+\frac{\beta+\delta}{N}}\\
&\qquad\leq C
\frac{1}{d_{n_k}^\delta}\|u_{n_k}\|_{H^1(\mathbb{R}^N)}^{2q}\to 0
\end{split}
\end{equation*}
as $k\to \infty$. Hence,
\begin{equation*}
\begin{split}
A_\rho&=\lim_{k\to\infty}\left(\frac{1}{2}\int_{\mathbb{R}^N}|\nabla
u_{n_k}|^2dx-\frac{\gamma}{2}\int_{\mathbb{R}^N}\frac{|u_{n_k}|^2}{|x|^\alpha}dx
-\frac{1}{2q}\int_{\mathbb{R}^N}(I_\beta\ast|u_{n_k}|^{q})|u_{n_k}|^{q}dx\right)\\
&\geq \liminf_{k\to
\infty}\left(\frac{1}{2}\int_{\mathbb{R}^N}|\nabla
u_{n_k}^{(1)}|^2dx-\frac{\gamma}{2}\int_{\mathbb{R}^N}\frac{|u_{n_k}^{(1)}|^2}{|x|^\alpha}dx
-\frac{1}{2q}\int_{\mathbb{R}^N}(I_\beta\ast|u_{n_k}^{(1)}|^{q})|u_{n_k}^{(1)}|^{q}dx\right)\\
&\quad+ \liminf_{k\to
\infty}\left(\frac{1}{2}\int_{\mathbb{R}^N}|\nabla
u_{n_k}^{(2)}|^2dx-\frac{\gamma}{2}\int_{\mathbb{R}^N}\frac{|u_{n_k}^{(2)}|^2}{|x|^\alpha}dx
-\frac{1}{2q}\int_{\mathbb{R}^N}(I_\beta\ast|u_{n_k}^{(2)}|^{q})|u_{n_k}^{(2)}|^{q}dx\right)\\
&\quad+\delta(\epsilon)\\
&\geq
A_{\bar{\sigma}}+A_{\bar{\eta}}^{\infty}+\delta(\epsilon)\ \
\mathrm{or}\ \
A_{\bar{\sigma}}^{\infty}+A_{\bar{\eta}}+\delta(\epsilon).
\end{split}
\end{equation*}
Letting $\epsilon\to 0$, we obtain that $A_\rho\geq
A_{\sigma}+A_{\rho-\sigma}^{\infty}$ or $A_\rho\geq
A_{\sigma}^{\infty}+A_{\rho-\sigma}$, which contradicts
(\ref{e33.8}).

\textbf{Step 5.} From Steps 3 and 4 and Lemma \ref{lem con-com},
similarly to Step 5 in the proof of Theorem  \ref{thm1.1}, we know
\begin{equation*}
u_{n_k}\to u\ \mathrm{in}\ L^s(\mathbb{R}^N)\ \mathrm{for}\ s\in
[2,2N/(N-2)_+)
\end{equation*}
and then
\begin{equation*}
\int_{\mathbb{R}^N}(I_\beta\ast|u_{n_k}|^{q})|u_{n_k}|^{q}dx\to\int_{\mathbb{R}^N}(I_\beta\ast|u|^{q})|u|^{q}dx.
\end{equation*}
Hence, $u$ is a minimizer of $A_\rho$ and  $u_{n_k}\to u$ in
$H^1(\mathbb{R}^N)$.

\textbf{Step 6.} Under Assumption A, if $1+\beta/N<q<1+(2+\beta)/N$
and $u_0\in H^1(\mathbb{R}^N)$ or $q=1+(2+\beta)/N$ and
$\|u_0\|_2<\|W_q\|_2$, then the Cauchy problem (\ref{e1.1}) admits a
global solution $u(t)\in C(\mathbb{R},H^1(\mathbb{R}^N))\cap
C^1(\mathbb{R},H^{-1}(\mathbb{R}^N))$ with $u(0)=u_0$.

\textbf{Step 7.} Similarly, $G_\rho$ is orbitally stable.  The proof
is complete.

\bigskip

\textbf{Proof of Theorem \ref{thm1.2}.} It can be done by small
modifications of the proof  of Theorem \ref{thm1.1}, and we omit it.

\textbf{Proof of Theorem \ref{thm1.4}.} It can be done by small
modifications of the proof  of Theorem \ref{thm1.3}, and we omit it.

\textbf{Proof of Theorem \ref{thm1.5}.} It can be done by combining
the proof of Theorems \ref{thm1.1} and  \ref{thm1.3}, and we omit
it.

\medskip

\textbf{Acknowledgements.} The authors would like to express sincere
thanks to the anonymous referee for his or her carefully reading the
manuscript and valuable comments and suggestions. This work was
supported by Tianjin Municipal Education Commission with the Grant
no. 2017KJ173 ``Qualitative studies of solutions for two kinds of
nonlocal elliptic equations".



\end{document}